\documentclass[10pt]{article}
\newtheorem{prelemmaa}{{\bf LEMMA}}

\newtheorem{prelem}{{\bf THEOREM}}

\newtheorem{preque}{{\bf QUESTION}}

\newtheorem{prelemma}{LEMMA}

\newtheorem{preproof}{{\bf PROOF.}}

\newenvironment{proof}[1]{\begin{preproof}{\rm
               #1}\hfill{\rule[-0.5mm]{2mm}{2mm}}}{\end{preproof}}

\newtheorem{preproposition}{{PROPOSITION}}

\newtheorem{preremark}{REMARK}

\newtheorem{precorollary}{{COROLLARY}}

\newtheorem{precorn}{{COROLLARY}}

\newtheorem{predefinition}{DEFINITION}

\newtheorem{preexample}{EXAMPLE}

\newtheorem{preconjecture}{{CONJECTURE}}

\newtheorem{pretheo}{{\bf THEOREM}}

\def\newpic#1{}
\date{}
\begin{document}
\title{\bf The Intersection  problem for $2$-$(v,5,1)$ directed block designs}
\author{
{\sc Nasrin Soltankhah}\footnote{Corresponding author: E-mail: soltan@alzahra.ac.ir
soltankhah.n@gmail.com} \sc  AND somaye ahmadi
 \\ [5mm]
Department of Mathematics\\ Alzahra University  \\
Vanak Square 19834 \ Tehran, I.R. Iran }
%
%\begin{document}
\maketitle
\begin{abstract}
%A $t$-$(v,k,\lambda)$ directed design (or simply a
%$t$-$(v,k,\lambda)DD$) is a pair  $(V,\beta)$, where $V$ is
%$v$-set, and $\beta$ is a collection of ordered $k-$tuples of
%distinct elements of $V$ (blocks), such that each ordered
%$t$-tuple  of distinct elements of $V$ appears in precisely
%$\lambda$ blocks. In this paper we determine the intersection
%spectrum for a pair of $2-(v,5,1)$ directed designs. D. J. Street
%and W. B. Wilson showed that the necessary and sufficient
%condition for the existence of a $2$-$(v,5,1)DD$ is $v\equiv1~or~
%5~(mod ~20)$ and $v \ne 15$, in 1980. Our study shows that for
%$v\equiv1~or~ 5~(\bmod~ 20)$ and $v \ne 11,15$, $I_D(v)=J_D(v)$
%and $\{0,1,2,3,11\}\subseteq I_D(11)$ .

%with one exception that
%$2-(15,5,1)DD$ does not exist $\cite{CJ,HJW}$.
The intersection problem for a pair of $2$-$(v,3,1)$ directed
designs and $2$-$(v,4,1)$ directed designs is solved by Fu in $1983$ and by Mahmoodian and Soltankhah  in
$1996$, respectively. In this paper we determine  the intersection
problem for $2$-$(v,5,1)$ directed designs.
 \end{abstract}
% \vspace{3mm}\\
%{\bf AMS subject classification: 05} \\
%
\hspace*{-2.7mm} {\bf KEYWORDS:} {  \sf Directed
designs,~Intersection of directed designs }
%
%%%%%%%%%%%%%%%%%%%%%%%%%%%%%%%%%%%%%%%%%%%%%%%%%%%%%%%%%%%%%%%%%%
\section{Introduction} %sec 1
\setcounter{theorem}{0}\setcounter{preproposition}{0}
\setcounter{precorollary}{0} \setcounter{prelemma}{0}
\setcounter{preexample}{0}
A  $t$-$(v,k,\lambda)$ directed design (or simply a
$t$-$(v,k,\lambda)DD$) is a pair$(V,\beta)$, where $V$ is a $v$-set,
and $\beta$ is a collection of ordered $k-$tuples of distinct
elements of $V$ (blocks), such that each ordered $t$-tuple of
distinct elements of $V$ appears in precisely $\lambda$ blocks.
We say that a $t$-tuple appears in a
$\it k$-tuple, if its components appear in that $k$-tuple as a set,
and they appear with the same order.
 For example, the 5-tuple $(0, 1, 4, 14, 16)$ contains the ordered pairs (0,1),
(0,4), (0,14), (0,16), (1,4), (1,14), (1,16), (4,14), (4,16), and (14,16).

The problem of determining the possible numbers of common blocks
between two designs with the same parameters is studied
extensively. Kramer and Mesner \cite{KM} asked the following: for
what values $s$ do there exist two Steiner systems $S(t,k,v)$
intersecting in $s$ blocks?  The spectrum of possible
intersection sizes for ordinary designs $S(2,3,v)$ and $S(2,4,v)$
was  settled by Lindner and Rosa $\cite{CA}$ and by Colbourn,
Hoffman and Lindner $\cite{CJC}$ respectively. H. L. Fu \cite{Fu}
discussed the intersection numbers of $S(3,4,v)$ for $v\equiv4~or~ 8(\bmod ~12)$, and
Hartman and Yehudai $\cite{HY}$ completed the determination of the spectrum of possible
intersection sizes for Steiner quadruple systems of all admissible orders $v$ except possibly $v=14,16$.
Lindner and Wallis $\cite{CW}$ and independently H. L. Fu
$\cite{HLF}$ settled the spectrum of possible intersection sizes
for a pair of  $2$-$(v,3,1)$DDs (transitive triple systems) for all admissible
$v$. The intersection problem  for a pair of $2$-$(v,4,1)$DDs and $3$-$(v,4,1)$DDs were
solved by Mahmoodian and Soltankhah $\cite{MS,MN}$.
%The last Shen
%\cite{HS} solved  the intersection problem for Kirkman triple
%systems.
 In this paper, we solve the intersection problem for
$2$-$(v,5,1)$DDs. The existence problem of $2$-$(v,5,\lambda)$DDs
has been solved in $\cite{DJW}$.  The necessary and sufficient
condition for the  existence of a $2-(v,5,1)DD$ is $v\equiv1~or~
5~(\bmod~ 10)$ with one exception that $2$-$(15,5,1)$DD does not
exist.
%the
%nonexistence of which is established by  Hall and Connor
%$\cite{CJ,HJW}$.

The number of blocks in a $2-(v,5,1)DD$ is equal to $ b_v  =
\frac{{v(v - 1)}}{{10}}$. Let $J_D (v) = \{ 0,1,...,b_v  - 2,b_v
\}$, and let  $I_D(v)$ denote  the set of all possible integers
$m$, such that there exist two $2-(v,5,1)DDs$ with exactly $m$
common blocks.~It is clear that $I_D(v)\subseteq J_D(v)$.

The notation is similar to that used in $\cite{hand}$. Let $K = \{
k_1 ,...,k_l \}$ be a set of positive integers. A pairwise balanced design
($~PBD(v,K,\lambda)$ ~or $(K,\lambda)PBD~$) of order $v$ with the
block sizes from $K$ is  a pair $(V,\beta)$, where $V$ is a
finite set of size $v$,  and $\beta$ is a family of subsets
(blocks) of $V$,   such that $(1)$ if $b\in \beta$  then $\left|
b \right| \in K$,   and $(2)$ every pair of distinct elements
of $V$ occurs in exactly $\lambda$ blocks of $\beta$. The
notations  $PBD(v,K)$ and $K-PBD$ of order $v$ are often used
when $\lambda=1$.

Let $K$ and $G$ be the  sets of positive integers, and let
$\lambda$ be a positive integer. A group divisible design of
index $\lambda$ and order $v$ $(~(K,\lambda)-GDD~)$ is a triple
$(V,G,\beta)$, where $V$ is a finite set of cardinality $v$, $G$
is a partition of $V$ into subsets (groups) whose size lie in $G$,
and $\beta$ is a family of subsets (blocks) of $V$ that satisfy
$(1)$ if $b \in \beta$ then $\left| b \right| \in K$, and  $(2)$
every pair of distinct elements of $V$ occurs in exactly $\lambda$
blocks or one group, but not both. If $v=a_1g_1+a_2g_2+...+a_sg_s$ and if there are $a_i$ groups
of size $g_i$, $i=1,2,...,s$ then the $(K,\lambda)-GDD$ is of the
type $g_1^{{a_1}}g_2^{{a_2}}...g_s^{{a_s}}$, or is of the type $M$, where $M=\{g_1,\cdots,g_s\}$.

A directed group divisible design $(~(K,\lambda)-DGDD~)$ is a
group divisible design $GDD$ in which every block is ordered and
each ordered pair formed from  distinct elements of different groups  occurs in exactly
$\lambda$ blocks.

In this paper, we extensively use the concept of ``trade"  defined
as follows.
 A $T(t,k,v)$ directed trade of volume $s$
consists of two disjoint collections
 $T_1$ and $T_2$, each of $s$ blocks,
  such that each ordered $t$-tuple occurs in the same number of blocks $T_1$
   as of $T_2$. It is usually denoted by $T=(T_1,T_2)$.
%Intersection  design is the number common blocks between two designs
%with the same set of parameters.

 Let $D=(V,\beta)$ be a  directed design  and $T=(T_1,T_2)$  be a
$T(v,k,\lambda)$ directed trade of volume $s$. If $T_1 \subseteq
\beta$, we say that $D$ contains the directed trade $T$, and if we
replace $T_2$ with $T_1$, then we  obtain a new design $D_1=(D \setminus T_1)\cup T_2$ which is denoted by $D_1=D+T$ with
the same parameters of $D$, and $\left| {D_1 \cap D } \right| =b_v-s$. This method of ``trade off" are  used frequently in this
paper.
%%%%%%%%%%%%%%%%%%%%%%%%%%%%%%%%%%%%%%%%%%%%%%%%%%%%%%%
\section{Some small cases}
\setcounter{theorem}{0}\setcounter{preproposition}{0}
\setcounter{precorollary}{0} \setcounter{prelemma}{0}
\setcounter{preexample}{0}
In this section we obtain  the intersection  size  of
$2$-$(v,5,1)$DDs, for $v=5,11,21,25$ that will be applied in proof
of  the  Theorems \ref{10},\ref{11} and \ref{13}.
\begin{prelemma}\label{8}
$I_D(5)=J_D(5)$.
\end{prelemma}
\begin{proof}
{Let $D_1: (0,1,2,3,4)~(4,3,2,1,0) $ and $ D_2:
(1,0,2,3,4)~(4,3,2,0,1)$ be two $2$-$(5,5,1)$DDs on the set
$\{0,1,2,3,4\}$. We have $ \left| {D_1 \cap D_1 } \right| =
2~,~\left| {D_1 \cap D_2 } \right| = 0$, so this results in
$I_D(5)=J_D(5)$.}
\end{proof}
\begin{prelemma}
$\{ 0,1,2,3,11\}   \subseteq I_D (11)$.
\end{prelemma}
\begin{proof}
{
Let $D_1$ be a $2$-$(11,5,1)$DD with the base block $(3,5,1,4,9)~(\mathrm{mod}
~11)$, and $D_2$ with the base block $(9,4,1,5,3)~(\mathrm{mod}~11)$ on the
set $\{0,1,...,9,10\}$, we have ${\left| {D_1  \cap D_1 } \right| =
11},~{\left| {D_1  \cap D_2 } \right| = 0}$.  Let $\alpha$ denotes
a permutation on the same set. For the following permutations on
the  elements of each block of $D_1$ and $D_2$ we have:
\begin{center}
$ \small
\begin{array}{cllllll}
  \alpha _1  =  & (39)(54)& & & \left| {D_2  \cap D_1 \alpha _1 } \right| = 1; & &\\
  & & & & & & \\
  \alpha _2  =  &(078) & & & \left| {D_1  \cap D_1 \alpha _2 } \right| = 2; & &\\
  & & & & & &\\
  \alpha _3  =  & (45)& & & \left| {D_1  \cap D_1 \alpha _3 } \right| = 3. &&
\end{array}$
\end{center}
This  results  in $\{ 0,1,2,3,11\}   \subseteq I_D (11)$.
}
\end{proof}
\begin{prelemma}
$I_D(21)=J_D(21)$.
\end{prelemma}
\begin{proof}
{
Let $D$ be a  $2$-$(21,5,1)$DD on the set $\{0,1,..,20\}$, with the
two  base blocks  $(0,1,6,8,18)$ and $(1,0,16,14,4)$, $(\mathrm{mod}~21)$. In
design $D$,  there exist  $21$ disjoint directed trades  of volume
$2$ and  at least a directed trade of volume $3$:
\begin{center}
$ \small
\begin{array}{cc}
   & T'_i:(i,1+i,6+i,8+i,18+i)~(1+i,i,16+i,14+i,4+i) \\
  T_i & \\
  &T''_i:(1+i,i,6+i,8+i,18+i)~(i,1+i,16+i,14+i,4+i) \\
   & 0 \le i \le 20
\end{array}
$
\end{center}
\begin{center}
$ \small
\begin{array}{ccccc}
  &  R' :(0,1,6,8,18)~(1,0,16,14,4)~(8,9,14,16,5)  \\
  R&\\
    &  R'' :(1,0,6,8,18)~(0,1,14,16,4)~(8,9,16,14,5)
\end{array}
$
\end{center}
Let $D_1=D+R$,~so we have:
\begin{center}
$ \small
\begin{array}{rlccc}
   \left| {D \cap D_1 } \right| =  & 39;& &&\\
   & & & &\\
   \left| {(D + \sum\limits_{i = 0}^I {T_i } ) \cap D} \right| =  &
    42 - (2(I+1))& &0 \le I \le 20;&  \\
   & & & & \\
   \left| {(D_1  + \sum\limits_{i = 1 }^{I} {T_i } ) \cap D} \right| =  &
   42 - (2I +3) & & 1 \le I \le 7;&  \\
   & & & & \\
   \left| {(D_1  + \sum\limits_{i = 1 }^I {T_i } ) \cap D} \right| =  &
   42 - (2I +1) & & 9 \le I \le 20.&
\end{array}
$
\end{center}
This results in  $I_D(21)=J_D(21)$.
}
\end{proof}
\begin{prelemma}
$I_D(25)=J_D(25)$.
\end{prelemma}
\begin{proof}
{
Let $D$ be $2$-$(25,5,1)$DD, on the set
$\mathrm{V}=\mathrm{Z_5}\times\mathrm{Z_5}$ that is obtained by  developing  the second coordinate of  below base blocks $(\mathrm{mod} ~5)$ except the first two blocks to which only the first coordinate should be expanded.
\begin{center}$$
\begin{array}{cc}
    {\small{((0,0),(0,1),(0,2),(0,3),(0,4))}} &  {\small{((0,4),(0,3),(0,2),(0,1),(0,0))}}    \\
                                          &                                           \\
   {\small{((0,0),(1,1),(2,4),(3,4),(4,1))}} & {\small{((4,0),(3,1),(2,3),(1,1),(0,0))}}      \\
  {\small{((0,3),(1,2),(2,3),(3,1),(4,1))}}   & {\small{((4,0),(3,3),(2,2),(1,2),(0,3))}}     \\
   {\small{((0,4),(1,1),(2,0),(3,1),(4,4))}}   & {\small{((4,3),(3,3),(2,4),(1,1),(0,4))}}    \\
  {\small{((0,4),(1,4),(2,1),(3,0),(4,1))}}   & {\small{((4,0),(3,2),(2,0),(1,4),(0,4))}}     \\
   {\small{((0,3),(1,1),(2,1),(3,3),(4,2))}}   &  {\small{((4,1),(3,0),(2,0),(1,1),(0,3))}}
\end{array}$$
\end{center}
\vspace{.2cm}
  Now we list some small $T(2,5,25)$ directed trades:
$$
\begin{array}{ccc}
  &T'_i :{\small{((i,0),(i,1),(i,2),(i,3),(i,4))((i,4),(i,3),(i,2),(i,1),(i,0))}}& \\
  T_i &  &  \\
    & T''_i:{\small{((i,0),(i,1),(i,2),(i,4),(i,3))((i,3),(i,4),(i,2),(i,1),(i,0))}}& \\
   & 1 \leq i \leq 5; &
\end{array}
$$
$$
\begin{array}{ccc}
  &T'_{5+i} :{\small{((0,i),(1,1+i),(2,4+i),(3,4+i),(4,1+i))((4,i),(3,1+i),(2,3+i),(1,1+i),(0,i))}}& \\
  T_{5+i} &  &  \\
    & T''_{5+i}:{\small{((1,1+i),(0,i),(2,4+i),(3,4+i),(4,1+i))((4,i),(3,1+i),(2,3+i),(0,i),(1,1+i))}}& \\
   & 1 \leq i \leq 5; &
\end{array}
$$
$$
\begin{array}{ccc}
  &T'_{10+i} :{\small{((0,3+i),(1,2+i),(2,3+i),(3,1+i),(4,1+i))((4,0+i),(3,3+i),(2,2+i),(1,2+i),(0,3+i))}}  & \\
  T_{10+i} &  &  \\
    & T''_{10+i}:{\small{((1,2+i),(0,3+i),(2,3+i),(3,1+i),(4,1+i))((4,i),(3,3+i),(2,2+i),(0,3+i),(1,2+i))}}  & \\
   & 1 \leq i \leq 5; &
\end{array}
$$
$$
\begin{array}{ccc}
  &T'_{15+i} :{\small{((0,4+i),(1,1+i),(2,i),(3,1+i),(4,4+i))((4,3+i),(3,3+i),(2,4+i),(1,1+i),(0,4+i))}} & \\
  T_{15+i} &  &  \\
    & T''_{15+i}:{\small{((1,1+i),(0,4+i),(2,i),(3,1+i),(4,4+i))((4,3+i),(3,3+i),(2,4+i),(0,4+i),(1,1+i))}} & \\
   & 1 \leq i \leq 5; &
\end{array}
$$
$$
\begin{array}{ccc}
  &T'_{20+i} :{\small{((0,4+i),(1,4+i),(2,1+i),(3,i),(4,1+i))((4,i),(3,2+i),(2,i),(1,4+i),(0,4+i))}} & \\
  T_{20+i} &  &  \\
    & T''_{20+i}:{\small{((1,4+i),(0,4+i),(2,1+i),(3,i),(4,1+i))((4,i),(3,2+i),(2,i),(0,4+i),(1,4+i))}} & \\
   & 1 \leq i \leq 5; &
\end{array}
$$
$$
\begin{array}{ccc}
  &T'_{25+i} : {\small{((0,1+i),(1,4+i),(2,4+i),(3,1+i),(4,i))((4,4+i),(3,3+i),(2,3+i),(1,4+i),(0,1+i))}}    & \\
 T_{25+i} &  &  \\
    &T''_{25+i}:{\small{((1,4+i),(0,1+i),(2,4+i),(3,1+i),(4,i))((4,4+i),(3,3+i),(2,3+i),(0,1+i),(1,4+i))}}  & \\
   & 1 \leq i \leq 5. &
\end{array}
$$
$$
\begin{array}{ccc}
  &L' : {\small{((0,0),(1,1),(2,4),(3,4),(4,1))((4,0),(3,1),(2,3),(1,1),(0,0))}}  &    \\
   & {\small{((4,0),(3,4),(2,4),(1,0),(0,2))}}  &  \\
  L &  &  \\
    & L:{\small{((1,1),(0,0),(3,4),(2,4),(4,1))((4,0),(3,1),(2,3),(0,0),(1,1))}}       & \\
   &{\small{((4,0),(2,4),(3,4),(1,0),(0,2))}}  &
\end{array}
$$
Let $D_1=D+L$,  so we have:
\begin{center}
$ \small
\begin{array}{rlccc}
   \left| {D \cap D_1 } \right| =  & 57;& & &  \\
   & & & &\\
   \left| {(D + \sum\limits_{i = 1}^I {T_i } ) \cap D} \right| =  & 60 - (2I) &  &{1 \le I \le 30};&  \\
   & & && \\
   \left| {(D_1  + \sum\limits_{i = 2}^I {T_i } ) \cap D} \right| =  & 60 - (2(I-1) + 3)& & 2 \le I \le 29 . & \\
\end{array}
$
\end{center}
This results in  $I_D(25)=J_D(25)$.
}
\end{proof}
%%%%%%%%%%%%%%%%%%%%%%%%%%%%%%%%%%%%%%%%%%%%%%%%%%%%%%%%%%%%%%%
\section{Recursive Construction} %sec 3
\setcounter{theorem}{0}\setcounter{preproposition}{0}
\setcounter{precorollary}{0} \setcounter{prelemma}{0}
\setcounter{preexample}{0}
In this section we introduce a Construction which will be applied
in constructing  of  directed designs with  required intersection sizes.

{\bf Construction. }\label{1}
 ~If there exists  a group divisible design $(G,\beta)$
of order $v$ with  block size $5,6$ and   groups each of  size
congruent to  $ ~0,2~(\bmod~5 )$,  then there exists a
$2$-$(2v+1,5,1)$DD.
\begin{proof}
{
Let $(G,\beta)$ be a group divisible design on the element set
$V$. We form a $2$-$(2v+1,5,1)$DD on the element set $(V\times Z_2)
\cup \{\infty\}$ as follows. For each block of size five $b\in
\beta$, say $b=(x_1,x_2,x_3,x_4,x_5)$, we form a $5$-$DGDD$ of type
$2^5$ on $b \times Z_2$, such that its groups are $\{x_1\} \times
Z_2,~\{x_2\} \times Z_2,~\{x_3\} \times Z_2,~\{x_4\} \times
Z_2,~\{x_5\} \times Z_2$ and for each block of  size six $b \in
\beta$, say $b=\{x_1,x_2,x_3,x_4,x_5,x_6\}$, we form a $5$-$DGDD$
of type $2^6$ on $b \times Z_2$, such that its groups are
$\{x_1\} \times Z_2,~\{x_2\} \times Z_2,~\{x_3\} \times
Z_2,~\{x_4\} \times Z_2,~\{x_5\} \times Z_2,\{x_6\} \times Z_2$.
These $DGDDs$ exist, as we will see later on. Finally for each
group $g$ of $G$, we substitute a $2-(2\left| g \right|+1,5,1)DD$
on $(g \times Z_2)\cup \{\infty\}$.
}\end{proof}
In applying  construction  and  future lemmas  we need some
$DGDDs$ and $GDDs$. We may use  $5$-$DGDDs$ ~of type $2^5$ and
$5$-$DGDDs$ of type $2^6$.
 Let $I_G(10)$ be the set of all
possible integers $m$, such that there exist two such $5$-$DGDDs$
of type  $2^5$ with the same groups and exactly $m$ common blocks
so let $I_{G}(12)$ be the set of all possible integers $m$, such
that there exist two such $5$-$DGDDs$ of type  $2^6$ with the same
groups and exactly $m$ common blocks.
\begin{prelemma}\label{2}
$I_{G}(10)=\{0,1,...,6,8\}$.
\end{prelemma}
\begin{proof}{
Let $G$ be a  $5$-$DGDD$~ of type  $2^5$  with  the following blocks
and groups.\\
 groups:$$\{1,2\},~\{3,4\},~\{5,6\},~\{7,8\},~\{0,9\}$$
 blocks:$$
\small
\begin{array}{cccc}
         (7,9,6,4,2)  & (2,3,6,9,8) & (5,4,8,1,9) & (2,4,7,0,5) \\
         (6,0,7,3,1) & (8,5,3,0,2) & (1,0,8,4,6) & (9,1,3,5,7) \\
       \end{array}
$$
 Now we list some small directed trades:

\vspace{.2cm}
 $
 \small
\begin{array}{ccccc}
  \underline{Directed~ trades} &  & \underline{Blocks~ removed} &  & \underline{Blocks ~added} \\
  T_1 &  & (7,9,6,4,2)~(2,3,6,9,8) &  & (7,6,9,4,2)~(2,3,9,6,8) \\
  T_2 &  & (5,4,8,1,9)~(1,0,8,4,6) &  & (5,8,4,1,9)~(1,0,4,8,6) \\
  T_3 &  & (2,4,7,0,5)~ (6,0,7,3,1) &  & (2,4,0,7,5)~ (6,7,0,3,1) \\
  T_4 &  & (8,5,3,0,2)~(9,1,3,5,7) &  & (8,3,5,0,2)~(9,1,5,3,7) \\
  T_5 &  & (8,5,3,0,2)~(9,1,3,5,7)~(2,4,7,0,5) &  & (8,3,0,5,2)~(9,1,5,3,7)~(2,4,7,5,0)
\end{array}$

\vspace{.1cm}

 Let $G_1=G+T_5$,~we have:
$$
\small
\begin{array}{rlcc}
  \left| {G  \cap G_1  } \right| = &  5;& &\\
   &  & & \\
   \left| {G  \cap (G  + \sum\limits_{i = 1}^I {T_i } )} \right| = & 8 - (2i)
    & {1 \le i \le 4};&\\
   &  & & \\
  \left| {G  \cap (G_1  + \sum\limits_{i = 1}^I {T_i } )} \right| = &  8 - (2i+3)
   & {1 \le i \le 2}. & \end{array}$$
\vspace{.1cm}
 This results in  $I_{G}(10)=\{0,1,...,6,8\}$.}
\end{proof}
\begin{prelemma}\label{3}
$I_{G}(12)=\{0,1,...,10,12\}$.
\end{prelemma}
\begin{proof}{
Let $G$ be a $5$-$DGDD$~ of type $2^6$ with the following blocks
and groups.\\
groups:$$\{1,2\},~\{3,4\},~\{5,6\},~\{7,8\},~\{9,10\},\\~\{0,11\}$$
blocks:
$$
\small
\begin{array}{cccc}
             (1,3,0,7,6)     & (2,4,11,8,5)& (0,10,5,1,4) & (6,8,4,0,9) \\
             (4,2,10,6,7)      & (8,6,1,10,11)      & (11,9,6,2,3)     & (9,11,7,4,1) \\
             (3,1,9,5,8)     & (7,5,2,9,0)      & (5,7,3,11,10)      & (10,0,8,3,2)
 \end{array}$$
 Now we
list some small directed trades:
\vspace{.1cm}

$ \small
\begin{array}{ccccc}
  \underline{Directed ~trades} &  & \underline{Blocks ~removed} &  & \underline{Blocks ~added} \\
  T_1 &  & (1,3,0,7,6)~(3,1,9,5,8)            &  & (3,1,0,7,6)~(1,3,9,5,8) \\
  T_2 &  & (4,2,10,6,7)~(2,4,11,8,5)            &  & (2,4,10,6,7)~(4,2,11,8,5) \\
  T_3 &  & (8,6,1,10,11)~(6,8,4,0,9)               &  & (6,8,1,10,11)~(8,6,4,0,9) \\
  T_4 &  & (7,5,2,9,0)~(5,7,3,11,10)              &  & (5,7,2,9,0)~(7,5,3,11,10) \\
  T_5 &  & (0,10,5,1,4)~(10,0,8,3,2)           &  & (10,0,5,1,4)~(0,10,8,3,2) \\
  T_6 &  & (11,9,6,2,3)~(9,11,7,4,1)              &  & (9,11,6,2,3)~(11,9,7,4,1) \\
  T_7 &  & (11,9,6,2,3)~(9,11,7,4,1)~(0,10,5,1,4) &  & (9,11,6,2,3)~(11,9,7,1,4)~(0,10,5,4,1) \\
\end{array}
$

\vspace{.1cm}

Let $G_1=G+T_7$,  we have:

\vspace{.1cm}

$ \small
\begin{array}{rlcc}
  \left| {G  \cap  G_1} \right| = & 9; & &\\
  & & & \\
   \left| {G  \cap (G  + \sum\limits_{i = 1}^I {T_i } )} \right| = & 12 - (2i)
    & {1 \le i \le 6};&\\
  & & & \\
  \left| {G  \cap (G_1 + \sum\limits_{i = 1}^I {T_i } )} \right| = & 12 - (2i+3) &
   {1 \le i \le 4}.&
\end{array}$

\vspace{.1cm}

This results  in $I_{G}(12)=\{0,1,...,10,12\}$.}
\end{proof}
\begin{prelemma}\label{4}
Let $(G,\beta)$ be a group divisible design of  order $v$ with
$r$ blocks of  size $5$ and $s$ blocks of  size $6$, and $p$
groups each of size congruent to  $ ~0,2~(\mathrm{\bmod}~5 )$. For $1 \le
i \le r$, let $a_i \in I_{G}(10)$; for $1 \le i \le s$, let $c_i
\in I_{G}(12)$; for $1 \le i \le p$, let $d_{i} \in I_D (2\left|
g \right| + 1)$. Then there exist two $2$-$(2v+1,5,1)$DDs
intersecting in precisely $$ \sum\limits_{j = 1}^{r } {a_j }  +
\sum\limits_{j = 1}^{s } {c_j } + \sum\limits_{i = 1}^p {d_i }$$
blocks.
\end{prelemma}
\begin{proof}{
Using construction, take two copies of the same group divisible
design $(G,\beta)$ and construct on them two $2$-$(2v+1,5,1)$DDs.
Corresponding to each  block of size $5$, say
$b_i=\{x_1,x_2,x_3,x_4,x_5\}$, for $1 \le i \le r$ place on $b_i
\times Z_2$ in the two systems $5-DGDs$ of type $2^5$ having the
same groups, and $a_i$ blocks in common. Corresponding to each
 block of size $6$, say $b_i=\{x_1,x_2,x_3,x_4,x_5,x_6\}$, for
$1 \le i \le s$  place on $b_i \times Z_2$ in the two systems
$5-DGDs$ of type $2^6$ having the same groups, and $c_i$ blocks
in common and corresponding to each  group $g$  of $G$ place two
$2$-$(2\left| g \right| + 1,5,1)$DDs with $d_i$ blocks in  common.
}\end{proof}
 We state following  Propositions which  provide the
bases for main results in the next section.
 \begin{preproposition}\label{5}
~{\rm \cite{FEB}}~ Let $v\geq 101$, and $v\equiv 1 ~or ~5 ~({\bmod}
~20)$, If $v\neq 141$, then $PBD(v,\{5,25^*\})$ exists.
\end{preproposition}
\begin{preproposition}\label{6}
~{\rm \cite{FEB}}~ Let $v\geq 85$, and $v\equiv 1~ or~ 5 ~({\bmod}
~20)$, If $v\neq 125$, then $PBD(v,\{5,21^\star \})$ exists.
\end{preproposition}
\begin{preproposition}\label{7}
~{\rm \cite{Hanani}}~ Let $u \equiv 0 ~or ~2 ~({\bmod} ~5)$, $~u \neq
7$, and $M=\{2,5,10,12,
15,17,20,22,32,35,\\37,40,42,45,47,50,52,55,57,67,
75,77,80,82,92,105,107, 110,112,115,117,120,122,132,167\}$,  then
there exists a $\{5,6\}-GDD$  of type $M$ and order $u$.
\end{preproposition}
\begin{preproposition}\label{14}
~{\rm \cite{BCM}}~There exists a $\{5,6\}$-$GDD$ of type $5^n $, for
all  $n \ge 5$, except possibly when $n \in \{7,8,10,16\}$.
\end{preproposition}
\begin{preproposition}\label{12}
~{\rm \cite{ES}}~There exists a $\{5,6\}$-$GDD$ of type $5^{4t + 1}
u$, for $0 \le u \le 5t$, for all
 $t \ge 1,~t \notin\{2,17,23,32\}$.
\end{preproposition}
%
%%%%%%%%%%%%%%%%%%%%%%%%%%%%%%%%%%%%%%%%%%%%%%%%%%%%%%%%%%%%%%
\section{Applying Construction }  %sec
\setcounter{theorem}{0}\setcounter{preproposition}{0}
\setcounter{precorollary}{0} \setcounter{prelemma}{0}
\setcounter{preexample}{0}
 In this section we prove the following  main Theorem.
 \begin{pretheo}\label{10}
For  $v\equiv 11~(\mathrm{\bmod} ~20)$, $v\neq 11,31,71$ , $I_D(v)=J_D(v)$.
\end{pretheo}
\begin{proof}{
According to  Proposition \ref{14} there exists a $\{ 5,6\}$-$GDD$
of type ${5^{{\textstyle{{v - 1} \over {10}}}}}$  for each of values
$v\equiv 11~(\bmod ~20)$ and  $v\neq 11,31,71$. We may apply
Construction and Lemma \ref{4} and  the fact that
$I_G(10)=\{0,1,...,6,8\}$,~$I_G(12)=\{0,1,...,10,12\}$ and
$\{0,1,2,3,11\}\subseteq I_D(11)$, we can deduce $I_D(v)=J_D(v)$,
for $v\equiv ~11~(\bmod~ 20)$, $v\neq 11,31,71$.
%all the required $DGDs$
%and $GDDs$ exist.
}\end{proof}
\begin{pretheo}\label{11}
For $v\equiv 1 ~or ~5~(\bmod ~20)$,~$v\geq 85$,~$I_D(v)=J_D(v)$.
\end{pretheo}
\begin{proof}
{For $v\geq 85,~v\ne 125$, according to  Propositions \ref{6}
 there exists a $PBD(v,\{5,21^\star\})$
. If we replace the block of size $21$ with a $2$-$(21,5,1)$DD and
 put a
$2$-$(5,5,1)$DD on each block of size $5$, then for all $v\equiv 1~
or~ 5~(\bmod~ 20)$, $v\geq 85,~v\ne 125$, we can obtain a
$2$-$(v,5,1)$DD. So according to  Propositions \ref{5} there exists
a $PBD(125,\{5,25^\star\})$ which if we replace the block of size
$25$ with a $2$-$(25,5,1)$DD and
 put a
$2$-$(5,5,1)$DD on each block of size $5$, then we  can obtain a
$2$-$(125,5,1)$DD. According to  the fact that $I_D(21)=J_D(21)$,
$I_D(25)=J_D(25)$ and $I_D(5)=J_D(5)$, therefore  we can deduce
$I_D(v)=J_D(v)$.
}\end{proof}
\begin{pretheo}\label{13}
For $v\equiv 15~(\bmod~ 20)$, $v \ge 175$ and $v\ne 215,235,335$,
$~I_D(v)=J_D(v)$.
\end{pretheo}
\begin{proof}{
According to   Proposition \ref{7}, for each of aforesaid values
there exists  a $\{5,6\}$-$GDD$ of  order  $~\frac{{v -
1}}{2}$ with  groups that their size belong to the set  $N$  which $N \subseteq (M \cap Z_{{\textstyle{{v - 3} \over
2}}})$.
 Now according to Construction,~Lemma \ref{4},~Theorems \ref{10} and \ref{11} and other results that we obtain in Appendix  (the fact that $I_{G}(10)=\{0,1,...,6,8\}$,~$I_{G}(12)=\{0,1,...,10,12\}$,
~$\{0,1,2,3,11\} \subseteq I_D(11)$,~$I_D(2n+1)=J_D(2n+1)$ for $n\in N$),~ we can
deduce $I_D(v)=J_D(v)$.
}\end{proof}
{\bf Main Theorem.}\\
 For $v \equiv 1,5~(\bmod ~10)$,~$v\ne
11,15$,~$I_D(v)=J_D(v)$  and $\{0,1,2,3,11\} \subseteq I_D(11)$.
\begin{proof}{
According to   Theorems \ref{10},\ \ref{11} and \ref{13}~ for ~$v
\equiv 1,5~(\bmod ~10)$,~$v\ne\{
31,35,41,45,55,61,65,71,\\75,81,95,115,135,215,235,335\}$,~$I_D(v)=J_D(v)$.
Theorem for  the remaining values of $v$ are proved in the
Appendix.}
\end{proof}
%
%%%%%%%%%%%%%%%%%%%%%%%%%%%%%%%
 \textbf{Acknowledgement:}
The authors wish to thank the referee who helped us to improve presentation
of this manuscript.
%%%%%%%%%%%%%%%%%%%%%%%%%%%%%%%%%%%%%%%%%%%%%%%%%%%%
%
%%%%%%%%%%%%%%%%%%%%%%%%%%%%%%%%%%%%%%%%%%%%%%%%%%%%%%%%%%%%%%%%%%%%%%%%%%%%%%%%%%%%%

%%%%%%%%%%%%%%%%%%%%%%%%%%%%%%%%%%%%%%%%%%%%%%%%%%%%%%%%%%%%%%%%%%%%%%
\section{Appendix}
In this section, we construct~ $2-(v,5,1)DDs$~ for  remain values
$v=31,35,41,45,55,61,65,71,75,81,$ \\
$95,115,135,155,215,235,335$,
and  by  method of  trade off we can obtain their intersection spectrum.\\\\
${\bullet ~~~I_D(31)=J_D(31)}$.\\
Let $D$ be a $2$-$(31,5,1)$DD on the set $\{0,1,...,30\}$, with base
blocks $(20,10,5,9,18)~ (6,12,24,3,17)$ \\ $ (1,16,8,2,4)~ (\bmod~
31)$. (This directed design is obtained from a $2$-$(31,5,2)$BD with
a suitable ordering on its base blocks \cite{Diff}). This
$2$-$(31,5,1)$DD contains $31$  disjoint directed  trades of volume $2$ and
$31$ disjoint directed  trades of volume $3$.

Some  $T(2,5,31)$ directed trades exist in $D$:
$$
 \small
\begin{array}{ccc}
        &    T'_i:(1+i,16+i,8+i,2+i,4+i)(2+i,8+i,20+i,30+i,13+i) &  \\
  T_i &    T''_i:(1+i,16+i,2+i,8+i,4+i)(8+i,2+i,20+i,30+i,13+i)&  \\
        &                 i=1,...,31; &
\end{array}$$
$$ \small
\begin{array}{ccc}
   & R'_i:(1+i,16+i,8+i,2+i,4+i)(2+i,8+i,20+i,30+i,13+i) & \\
  & (30+i,20+i,15+i,19+i,28+i) &  \\
 R_i &  &  \\
   & R''_i:(1+i,16+i,2+i,8+i,4+i)(8+i,2+i,30+i,20+i,13+i) &  \\
   & (20+i,30+i,15+i,19+i,28+i) &  \\
   & i=1,...,31. &  \\
\end{array}
$$
%
 %We apply  two groups trades $T_i$ and $R_i$, and
%obtain the set of intersections $(31,5,1)DD$.
 Let $D_1=D+R_1$ so we have:~$\left| {D \cap D_1 } \right| = 90$.
\vspace{.1cm}

Odd intersections:
\vspace{.1cm}
$$
\begin{array}{llrrr}
 {\left| {D \cap (D + \sum\limits_{i = 1}^I {{T_i}} )} \right| = 93 - 2I} & {} & {} & {1 \le I \le 31;}
  & {}  \\
 {} & {} & {} & {} & {}  \\
 {\left| {D \cap (D + \sum\limits_{i = 1}^{20} {{R_i} + } \sum\limits_{i = 21}^I {{T_i}} )} \right| =
 93 - (60 + 2(I - 20))} & {} & {} & {21 \le I \le 31;} & {}  \\
  {} & {} & {} & {} & {}  \\
  {\left| {D \cap (D + \sum\limits_{i = 1}^{28} {{R_i} + } \sum\limits_{i = 28}^I {{T_i}} )} \right| =
   93 - (84 + 2(I - 28))} & {} & {} & {28 \le I \le 31;} & {}  \\
  {} & {} & {} & {} & {}  \\
  {\left| {D \cap (D + \sum\limits_{i = 1}^{30} {{R_i} + } {T_{31}} )} \right| =
  1.} & {} & {} & {} & {}
\end{array}$$
Even intersections:
\vspace{.3cm} $$
\begin{array}{llrrr}
 {\left| {D \cap (D_1 + \sum\limits_{i = 2}^I {{T_i}} )} \right| =90 - 2(I-1) } & {}
  & {} & {2 \le I \le 31;} & {}  \\
 {} & {} & {} & {} & {}  \\
 {\left| {D \cap (D + \sum\limits_{i = 1}^{21} {{R_i} + } \sum\limits_{i = 21}^I {{T_i}} )} \right| =
  93 - (63 + 2(I - 20))} & {} & {} & {21 \le I \le 31;} & {}  \\
 {} & {} & {} & {} & {}  \\
 {\left| {D \cap (D + \sum\limits_{i = 1}^{27} {{R_i} + } \sum\limits_{i = 27}^I {{T_i}} )} \right| =
  93 - (81 + 2(I - 26))} & {} & {} & {27 \le I \le 31;} & {}  \\
 {} & {} & {} & {} & {}  \\
 {\left| {D \cap (D + \sum\limits_{i = 1}^{29} {R_i  + \sum\limits_{i
= 29}^I {T_i } } )} \right| =93 - (87 + 2(I - 28))} & { }&{}&{29
\le
I \le 31;}&{} \\
 {} & {} & {} & {} & {}  \\
 {\left| {D \cap (D + \sum\limits_{i = 1}^{31} {R_i } )} \right| =0. } & {} & {} & {} & {}
\end{array}$$
\vspace{.2cm}
 %This result that $I_D(31)=J_D(31)$.\\\\
${\bullet ~~~ I_D(35)=J_D(35)}$.\\
Let $D$ be a  $2$-$(35,5,1)$DD on the set $\{0,1,...,34\}$, is obtained by developing  the following
base blocks  under
the group generated by $(0)(1,2,...,17)(18,...,34)$,
$$
\begin{array}{lll}
  (31,2,9,30,34) & (7,29,34,26,1) & (18,25,1,31,5) \\
  (4,6,1,9,2) & (24,14,3,26,34) & (19,13,20,5,31) \\
  (20,2,0,1,18). &  &
\end{array}
$$
(This directed design is obtained from a super simple
$2$-$(35,5,2)$BD with a suitable ordering on its base blocks
\cite{GKL}). In design $D$ there exist $59$  disjoint directed
trades of volume $2$ and at least a directed trade of volume $3$.

 Some small $TD(2,5,35)$ directed trades exist in $D$:
\vspace{.1cm}
 $$
\begin{array}{lcr}
   & T'_{1i}:(31+i,2+i,9+i,30+i,34+i)(4+i,6+i,1+i,9+i,2+i) &  \\
  T_{1i} &  & ~~1\leq i \leq 17; \\
   & T''_{1i}:(31+i,9+i,2+i,30+i,34+i)(4+i,6+i,1+i,2+i,9+i) &
\end{array}
$$
\vspace{.1cm}
 $$
\begin{array}{lcr}
   & T'_{2i}:(7+i,29+i,34+i,26+i,1+i)(24+i,14+i,3+i,26+i,34+i) &  \\
  T_{2i} &  & 1\leq i \leq 17; \\
   & T''_{2i}:(7+i,29+i,26+i,34+i,1+i)(24+i,14+i,3+i,34+i,26+i) &
\end{array}
$$
$$
\begin{array}{lcr}
   & T'_{3i}:(18+i,25+i,1+i,31+i,5+i)(19+i,13+i,20+i,5+i,31+i) &  \\
  T_{3i} &  & 1\leq i \leq 17; \\
   & T''_{3i}:(18+i,25+i,1+i,5+i,31+i)(19+i,13+i,20+i,31+i,5+i) &
\end{array}
$$
\vspace{.1cm}
$$
\begin{array}{lcr}
   & T'_{4i}:(20+i,2+i,0,1+i,18+i)(21+i,3+i,0,2+i,19+i) &  \\
  T_{4i} &  & i=2I,~1\leq I \leq 8. \\
   & T''_{4i}:(20+i,0,2+i,1+i,18+i)(21+i,3+i,2+i,0,19+i) &
\end{array}
$$
$$
\begin{array}{lcr}
   & K':(31,2,9,30,34)(4,6,1,9,2)(9,31,19,28,3) &  \\
  K &  &  \\
   & K'':(9,31,2,30,34)(4,6,1,2,9)(31,9,19,28,3) &  \\
\end{array}
$$
\vspace{.1cm}
 Let $\alpha$ be the following permutation on the set
$\{0,1,...,34\}$,
$$
\alpha  = \left\{ {\begin{array}{*{20}c}
   {(m,m + 17)} & {} & {1 \le m \le 17}  \\
   {} & {} & {}  \\
   {(m,m - 17)} & {} & {18 \le m \le 34}
\end{array}} \right.$$
so we have $\left| {D \cap D\alpha } \right| = 0$. Therefore by
help of  directed trades that there exist in design $D$ we can obtain
the set of  intersections $2$-$(35,5,1)$DD and can  deduce $I_D(35)=J_D(35)$.\\

$\bullet~~~I_D(41)=J_D(41)$.\\
 Let $D$  be a  $2-(41,5,1)DD$ on
the set $V=\{0,1,...,40\}$, with the below  base blocks \cite{Diff}.
$$\begin{array}{ll}
   (1,37,18,16,10) & (2,33,36,32,20) \\
   (13,18,37,39,4) & (33,2,40,3,15)
 \end{array}$$
   In  design $D$ there exist $82$  disjoint directed  trades of volume $2$ and at least a directed trade $K$ of
volume $3$.
\vspace{.1cm}
 $$
 \small
\begin{array}{lcc}
   & T'_{1i}:(1+i,37+i,18+i,16+i,10+i)(13+i,18+i,37+i,39+i,4+i) &  \\
  T_{1i} &  &  \\
   & T_{1i}'':(1+i,18+i,37+i,16+i,10+i)(13+i,37+i,18+i,39+i,4+i) &  \\
   & 0\leq i \leq 40; &  \\
\end{array}
$$
\vspace{.1cm}
 $$
 \small
\begin{array}{lcc}
   & T'_{2i}:(2+i,33+i,36+i,32+i,20+i)(33+i,2+i,40+i,3+i,15+i) &  \\
  T_{2i} &  &  \\
   & T_{2i}'':(33+i,2+i,36+i,32+i,20+i)(2+i,33+i,40+i,3+i,15+i) &  \\
   & 0\leq i \leq 40; &
\end{array}
$$
\vspace{.1cm}
 $$
 \small
\begin{array}{lcc}
   & K':(1,37,18,16,10)(13,18,37,39,4)(25,30,8,10,16) &  \\
  K &  &  \\
   & K'':(1,18,37,10,16)(13,37,18,39,4)(25,30,8,16,10) &
\end{array}
$$
 Therefore by help of the above  directed trades,  we can
obtain
the set of  intersections $2$-$(41,5,1)$DD and  can   deduce $I_D(41)=J_D(41)$.\\

${\bullet~~~I_D(45)=J_D(45)}$.\\
Let $D$ be a  $2$-$(45,5,1)$DD on the set $\{0,1\} \times Z_{22}
\cup \{\infty\}$,  is obtained by developing the second
coordinate of following $9$ base blocks $(\bmod ~22)$.(This directed
design is obtained from a super simple $2$-$(45,5,2)$BD with a
suitable ordering on its base blocks \cite{GKL} )
$$ \small
\begin{array}{ll}
{((0,4),(1,9),(1,8),(1,0),(1,2))}&{((1,4),(0,0),(1,13),(1,8),(1,9))} \\
   {((1,21),(0,18),(0,21),(0,6),(0,20))}&{((0,0),(1,15),(0,18),(1,21),(0,8))} \\
   {((0,0),(0,16),(1,11),(0,9),(0,13))} & {((0,11),(1,1),(1,21),(0,9),(1,11))}\\
   {((0,6),(1,0),\infty ,(1,7),(0,1))} & {((1,0),(0,6),(0,11),(1,3),(0,17))} \\
   {((1,10),(1,18),(0,0),(0,1),(1,7))} &
\end{array}$$
Some small $T(2,5,45)$ directed trades exist in $D$:

\vspace{.1cm}
 $ \small
\begin{array}{lrcc}
   & {T'_{1i}:}&{((1,4+i),(0,i),(1,13+i),(1,8+i),(1,9+i))} &\\
  & &{((0,4+i),(1,9+i),(1,8+i),(1,i),(1,2+i))} &  \\
  {T_{1i}} &  & &{} \\
  & {T''_{1i}}:& {((1,4+i),(0,i),(1,13+i),(1,9+i),(1,8+i))}& \\
   & &{((0,4+i),(1,8+i),(1,9+i),(1,i),(1,2+i))}  & \\
   & & 1\leq i \leq 22; & \\
   &  &
\end{array}
$
 \vspace{.1cm}

 $ \small
\begin{array}{lrcc}
   & {T'_{2i}:} & {((0,i),(1,15+i),(0,18+i),(1,21+i),(0,8+i))} &\\
   & & {((1,21+i),(0,18+i),(0,21+i),(0,6+i),(0,20+i))} & \\
  {T_{2i}} &  & &{} \\
   & {T''_{2i}:} & {((0,i),(1,15+i),(1,21+i),(0,18+i),(0,8+i))}&\\
   & & {((0,18+i),(1,21+i),(0,21+i),(0,6+i),(0,20+i))} &  \\
   & & {1\leq i \leq 22}; &
\end{array}
$
 \vspace{.1cm}

 $ \small
\begin{array}{lrcc}
   & {T'_{3i}:}& {((0,i),(0,16+i),(1,11+i),(0,9+i),(0,13+i))} &\\
   & &{((0,11+i),(1,1+i),(1,21+i),(0,9+i),(1,11+i))}   & \\
  {T_{3i}} &  & &{} \\
   & {T''_{3i}:} &{((0,i),(0,16+i),(0,9+i),(1,11+i),(0,13+i))}&\\
   & & {((0,11+i),(1,1+i),(1,21+i),(1,11+i),(0,9+i))}   & \\
   & & 1\leq i \leq 22; &
\end{array}
$
 \vspace{.1cm}

 $ \small
\begin{array}{lrcc}
  & &{((0,6+i),(1,i),\infty ,(1,7+i),(0,1+i))}&\\
  &R'_i:&{((1,i),(0,6+i),(0,11+i),(1,3+i),(0,17+i))} &  \\
   & & {((1,10+i),(1,18+i),(0,i),(0,1+i),(1,7+i))} &  \\
  R_i & & &  \\
  & & {((1,i),(0,6+i),\infty ,(0,1+i),(1,7+i))}&\\
 &R''_i: & {((0,6+i),(1,i),(0,11+i),(1,3+i),(0,17+i))} &  \\
  & & {((1,10+i),(1,18+i),(0,i),(1,7+i),(0,1+i))}&  \\
   & & 1\leq i \leq 22. &
\end{array}
$

\vspace{.1cm}
 Therefore by help of the above directed  trades that there exist  in  directed design $D$, we can
obtain the set of  intersections design $D$ and  we  can   deduce $I_D(45)=J_D(45)$.\\\\
${\bullet~~~I_D(55)=J_D(55)}$.\\
Let $D$ be a  $2$-$(55,5,1)$DD on the set  $V = {Z_{54}} \cup
\left\{ \infty  \right\}$, is obtained by developing the below
blocks
 $+2~ (\bmod~ 54)$(This directed design is obtained from a super
simple $2$-$(55,5,2)$BD with a suitable ordering on  its base
blocks \cite{AB} ).
 $$
 \small
\begin{array}{llllll}
   {(27,16,\infty ,38,25)} & {(2,16,7,31,22)} & {(31,27,0,37,24)} & {(3,49,11,23,25)}&  \\
   {(0,39,25,38,19)} & {(25,10,3,28,7)} & {(0,27,4,53,17)} & {(0,23,11,30,40)} & \\
   {(1,36,38,39,34)}&{(0,7,28,16,8)}&{(32,10,22,53,4)} &&\\
\end{array}$$
 Some small $T(2,5,55)$ directed trades exist in $D$:
\vspace{.1cm}

 $
 \small
\begin{array}{lcc}
   {} & {{{R'}_{1i}}:(27 + 2i,16 + 2i,\infty ,38 + 2i,25 + 2i)(2i,39+2i,25+2i,38+2i,19+2i)} & {}  \\
   {} & {(1+2i,36+2i,38+2i,39+2i,34+2i)} & {}  \\
   {{R_{1i}}} & {} & {}  \\
   {} & {{{R'}_{1i}}:(27 + 2i,16 + 2i,\infty ,25 + 2i,38 + 2i)(2i,38+2i,39+2i,25+2i,19+2i)} & {}  \\
   {} & {(1+2i,36+2i,39+2i,38+2i,34+2i)} & {}  \\
    {} & {0 \le i \le 26}; & {}
\end{array}$
\vspace{.1cm}

$
 \small
\begin{array}{lcc}
   {} & {{{R'}_{2i}}:(2 + 2i,16 + 2i,7 + 2i,31 + 2i,22 + 2i)(25 + 2i,10 + 2i,3 + 2i,28 + 2i,7 + 2i)} & {}  \\
   {} & {(2i,7 + 2i,28 + 2i,16 + 2i,8 + 2i)} & {}  \\
   {{R_{2i}}} & {} & {}  \\
   {} & {R''{:_{2i}}(2 + 2i,7 + 2i,16 + 2i,31 + 2i,22 + 2i)(25 + 2i,10 + 2i,3 + 2i,7 + 2i,28 + 2i)} & {}  \\
   {} & {(2i,28 + 2i,7 + 2i,16 + 2i,8 + 2i)} & {}  \\
    {} & {0 \le i \le 26}; & {}
\end{array}$
\vspace{.1cm}

$
 \small
\begin{array}{lcc}
   {} & {{{R'}_{3i}}:(31 + 2i,27 + 2i,2i,37 + 2i,24 + 2i)(2i,27 + 2i,4 + 2i,53 + 2i,17 + 2i)} & {}  \\
   {} & {(32 + 2i,10 + 2i,22 + 2i,53 + 2i,4 + 2i)} & {}  \\
   {{R_{3i}}} & {} & {}  \\
   {} & {{{R''}_{3i}}:(31 + 2i,2i,27 + 2i,37 + 2i,24 + 2i)(27 + 2i,2i,53 + 2i,4 + 2i,17 + 2i)} & {}  \\
   {} & {(32 + 2i,10 + 2i,22 + 2i,4 + 2i,53 + 2i)} & {}  \\
    {} & {0 \le i \le 26}; & {}
\end{array}
$
\vspace{.1cm}
 $$
 \small
\begin{array}{lcc}
   {} & {{{T'}_i}:(3 + 2i,49 + 2i,11 + 2i,23 + 2i,25 + 2i)(2i,23 + 2i,11 + 2i,30 + 2i,40 + 2i)} & {}  \\
   T_i & {} & {}  \\
   {} & {{{T''}_i}:(3 + 2i,49 + 2i,23 + 2i,11 + 2i,25 + 2i)(2i,11 + 2i,23 + 2i,30 + 2i,40 + 2i)} & {}  \\
   {} & {0 \le i \le 26}. & {}
\end{array}
$$
\vspace{.1cm}
 Therefore by help of directed trades that there exist  in directed  design $D$, we can
obtain
the set of  intersections design $D$ and  we  can   deduce $I_D(55)=J_D(55)$.\\\\
${\bullet ~~~ I_D(61)=J_D(61)}$.\\
 Let $D$ be a  $2$-$(61,5,1)$DD on
the set $V=\{0,1,...,60\}$, with  the below base blocks \cite{Diff}.
$$\begin{array}{ccc}
  (0,4,23,9,45) & (0,55,37,44,29) & (0,60,48,58,27) \\
 (4,0,42,56,20) & (55,0,18,11,26) & (60,0,12,2,33)
\end{array}
$$
In directed  design $D$ there exist $183$  disjoint directed  trades of volume $2$ and  least a directed trade $K$ of
volume $3$.

$ \small
\begin{array}{lrc}
   {} & {T'_{1i}:} & {(i,4+i,23+i,9+i,45+i)(4+i,i,42+i,56+i,20+i)}  \\
   T_{1i} & {} & {}  \\
   {} & {T''_{1i}:} & {(4+i,i,23+i,9+i,45+i)(i,4+i,42+i,56+i,20+i)}  \\
    & {} & {0\leq i \leq 60};
\end{array}$
\vspace{.1cm}

$ \small
\begin{array}{lrc}
   {} & {T'_{2i}:} & {(i,55+i,37+i,44+i,29+i)(55+i,i,18+i,11+i,26+i)}  \\
   T_{2i} & {} & {}  \\
   {} & {T''_{2i}:} & {(55+i,i,37+i,44+i,29+i)(i,55+i,18+i,11+i,26+i)}  \\
    & {} & {0\leq i \leq 60};
\end{array}$
\vspace{.1cm}

$ \small
\begin{array}{lrc}
   {} & {T_{3i}':} & {(i,60+i,48+i,58+i,27+i)(60+i,i,12+i,2+i,33+i)}  \\
   T_{3i} & {} & {}  \\
   {} & {T_{3i}'':} & {(60+i,i,48+i,58+i,27+i)(i,60+i,12+i,2+i,33+i)}  \\
    & {} & {0\leq i \leq 60};
\end{array}$
\vspace{.1cm}

$ \small
\begin{array}{lrc}
   {} & {K':} &  {(0,60,48,58,27)(60,0,12,2,33)(24,25,37,27,58)} \\
   K & {} & {}  \\
   {} & {K'':} &  {(60,0,48,27,58)(0,60,12,2,33)(24,25,37,58,27)}
\end{array}$

Therefore by help of  the above  three groups of directed  trades, we can
obtain
the  intersections set  $2$-$(61,5,1)$DD and  can   deduce $I_D(61)=J_D(61)$.\\\\
${\bullet ~~~ I_D(65)=J_D(65)}$.\\
 Let $D$ be a   $2$-$(65,5,1)$DD on
the set $\mathrm{V}=\mathrm{Z}_{54}\cup \{{\infty_0,...,\infty_{10}}\}$, that is obtained by developing the below base blocks  $+2~(\mathrm{mod~54})$, that  for $z=0,3$ replace $\infty_z$ by $\infty_{z+x}~(x=1,2)$ when adding any value $\equiv 2x~(\mathrm{mod~6})$ to a base block. Finally form a $2$-$(11,5,1)$DD on the set $\{\infty_0,...\infty_{10}\}$.
$$
  \begin{array}{lllll}
     (31,0,35,41,20)& (3,1,41,15,35) & (4,9,2,\infty_3,27)&(\infty_0,29,1,28,30) &(41,8,\infty_6,29,0) \\
    (0,31,\infty_3,8,12) & (1,3,52,47,\infty_0) & (11,\infty_3,19,28,35)& (12,25,\infty_7,28,1) &(47,12,\infty_9,0,29) \\
    (0,14,\infty_0,32,15) & (20,14,44,0,53) & (0,44,28,50,47) & (3,8,\infty_{10},53,44) & (31,42,\infty_8,53,14)
  \end{array}$$
(This directed design is obtained from a super
simple $2$-$(65,5,2)$BD with a suitable ordering on  its base
blocks  \cite{AB}). In  design $D$ there exist $135$  disjoint directed  trades of volume $2$ and $27$ disjoint directed trades of volume $5$  and least a directed trade of
volume $3$.
$$\begin{array}{lll}
   & T'_{1i}:(31+2i,2i,35+2i,41+2i,20+2i)(2i,31+2i,\infty_3,8+2i,12+2i) &\\
  T_{1i} & & \\
   & T''_{1i}:(2i,31+2i,35+2i,41+2i,20+2i)(31+2i,2i,\infty_3,8+2i,12+2i)& \\
   & 1 \leq i \leq 27;&
\end{array}
 $$
$$\begin{array}{lll}
   & T'_{2i}:(3+2i,1+2i,41+2i,15+2i,35+2i)(1+2i,3+2i,52+2i,47+2i,\infty_{0+x}) &\\
  T_{2i} & & \\
   & T''_{2i}:(1+2i,3+2i,41+2i,15+2i,35+2i)(3+2i,1+2i,52+2i,47+2i,\infty_{0+x})& \\
   & 1 \leq i \leq 27;&
\end{array}
 $$
$$\begin{array}{lll}
   & T'_{3i}:(30+2i,35+2i,28+2i,\infty_{3+x},53+2i)(11+2i,\infty_{3+x},19+2i,28+2i,35+2i) &\\
  T_{3i} & & \\
   & T''_{3i}:(30+2i,28+2i,35+2i,\infty_{3+x},53+2i)(11+2i,\infty_{3+x},19+2i,35+2i,28+2i)& \\
   & 1 \leq i \leq 27;&
\end{array}
 $$
$$\begin{array}{lll}
   & T'_{4i}:(\infty_{0+x},29+2i,1+2i,28+2i,30)(12+2i,25+2i,\infty_7,28+2i,1+2i) &\\
  T_{4i} & & \\
   & T''_{4i}:(\infty_{0+x},29+2i,28+2i,1+2i,30)(12+2i,25+2i,\infty_7,1+2i,28+2i)& \\
   & 1 \leq i \leq 27;&
\end{array}
 $$
$$\begin{array}{lll}
   & T'_{5i}:(41+2i,8+2i,\infty_6,29+2i,2i)(47+2i,12+2i,\infty_9,2i,29+2i) &\\
  T_{5i} & & \\
   & T''_{5i}:(41+2i,8+2i,\infty_6,2i,29+2i)(47+2i,12+2i,\infty_9,29+2i,2i)& \\
   & 1 \leq i \leq 27;&
\end{array}
 $$
$$\begin{array}{lll}
   &(2i,14+2i,\infty_{0+x},32+2i,15+2i) (20+2i,14+2i,44+2i,2i,53+2i)  &\\
  & R'_{i}:(2i,44+2i,28+2i,50+2i,47+2i) (3+2i,8+2i,\infty_{10},53+2i,44+2i)  \\
  &(31+2i,42+2i,\infty_8,53+2i,14+2i)& \\
  R_{i} &  & \\
   &(14+2i,2i,\infty_{0+x},32+2i,15+2i) (20+2i,2i,53+2i,14+2i,44+2i)  &\\
  & R'_{i}:(44+2i,2i,28+2i,50+2i,47+2i) (3+2i,8+2i,\infty_{10},44+2i,53+2i)  \\
  &(31+2i,42+2i,\infty_8,14+2i,53+2i)& \\
   & 1 \leq i \leq 27.&
\end{array}
 $$
\vspace{.1cm}
$$\begin{array}{lll}
   & K':(0,14,\infty_0,32,15)  (20,14,44,0,53)  (0,44,28,50,47) &  \\
  K &  &  \\
   & K'':(14,0,\infty_0,32,15)  (20,0,14,44,53)  (44,0,28,50,47) &
\end{array}$$
Therefore according to  the six groups of disjoint directed trades
that there exist in design $D$, we can obtain
the intersection set of design  $2$-$(65,5,1)$DD and   can   deduce $I_D(65)=J_D(65)$.\\
\\
${\bullet ~~~ I_D(71)=J_D(71)}$.\\
Let $D$ be a  $2$-$(71,5,1)$DD on the set $\{0,1,...,70\}$, with the
following base blocks $(\bmod ~71)$.
$$
\begin{array}{ccc}
  (28,45,14,15,20) & (20,4,15,29,3) & (65,63,4,20,54) \\
  (58,40,6,8,30)&(50,30,8,11,26) &  \\
  (43,37,2,50,10)&(50,2,35,6,70) & \\
\end{array}
$$
 (This directed
design is obtained from a $2$-$(71,5,2)$BD with a suitable ordering
on its base blocks \cite{Diff}). In  design $D$, there exist $71$ disjoint directed trades of volume $3$ and  $142$
disjoint directed trades of volume $2$:
%Some small $T(2,5,71)$ directed trades exist in $D$:
%
\vspace{.4cm}

$
 \small
\begin{array}{rc}
   {} & {R'_i :}  {(28+i,45+i,14+i,15+i,20+i) (20+i,4+i,15+i,29+i,3+i)}\\
      &           (65+i,63+i,4+i,20+i,54+i)\\
   {R_{i} } & {}   \\
   {} & {R''_i :}  {(28+i,45+i,14+i,20+i,15+i) (4+i,15+i,20+i,29+i,3+i)}  \\
      &           (65+i,63+i,20+i,4+i,54+i)\\
   {} &  {0 \le i \le 70};
\end{array}
$
\vspace{.3cm}

$
 \small
\begin{array}{rc}
   {} & {T'_{1i} :}  {(58+i,40+i,6+i,8+i,30+i)(50+i,30+i,8+i,11+i,26+i)}  \\
   {T_{1i} } & {}   \\
   {} & {T''_{1i} :} {(58+i,40+i,6+i,30+i,8+i)(50+i,8+i,30+i,11+i,26+i)}  \\
   {} &  { 0 \le i \le 70};
\end{array}
$
\vspace{.3cm}

 $
 \small
\begin{array}{rc}
   {} & {T'_{2i} :}  {(43+i,37+i,2+i,50+i,10+i)(50+i,2+i,35+i,6+i,70+i)} \\
   {T_{2i} } & {}   \\
   {} & {T''_{2i} :}  {(43+i,37+i,50+i,2+i,10+i)(2+i,50+i,35+i,6+i,70+i)}\\
   {} &  {0 \le i \le 70}.
\end{array}
$
\vspace{.1cm}

 Therefore by help of the above three groups of directed  trades, we can
obtain
the  intersection  set  of design $D$ and  deduce $I_D(71)=J_D(71)$.\\\\
${\bullet ~~~ I_D(75)=J_D(75)}$.\\
Let D be a $2$-$(75, 5, 1)$DD on the set $V = Z_{74}  \cup \{
\infty\}$, is obtained by developing the below base  blocks $+2$
$(\bmod~74)$. (This directed design is obtained from a super simple
$2$-$(75,5,2)$BD with a suitable ordering on its base blocks
\cite{KCRW} ).
\vspace{.3cm}

$
 \small
\begin{array}{ccccc}
   {(6,0,47,44,3)} & {(40,65,26,3,44)} & {(64,29,52,63,0)}& {(2,15,\infty ,59,0)} & {(22,15,2,42,34)} \\
   {(49,32,0,63,6)} & {(48,27,23,17,65)} & {(17,23,5,60,56)} & {(25,3,27,43,1)} & {(0,59,50,69,58)}  \\
   {(27,48,0,30,2)}& {(46,1,6,61,70)} & {(60,7,61,0,65)}&{(22,58,45,67,0)}& {(1,2,37,9,29)}
\end{array}$
\vspace{.3cm}

This $2$-$(75, 5, 1)$DD contain $111$ disjoint directed  trades of volume $2$ and $111$ disjoint directed
trades  of volume $3$.
 Now list these $T(2, 5, 75)$ directed trades:
\vspace{.3cm}

$
 \small
\begin{array}{llc}
   {} & {T'_{1i} :} & {(6 + 2i,2i,47 + 2i,44 + 2i,3 + 2i)(40 + 2i,65 + 2i,26 + 2i,3 + 2i,44 + 2i)}  \\
   {T_{1i} } & {} & {}  \\
   {} & {T''_{1i} :} & {(6 + i2,2i,47 + 2i,3 + 2i,44 + 2i)(40 + 2i,65 + 2i,26 + 2i,44 + 2i,3 + 2i)}  \\
   {} &  & {0\le i \le 36};
\end{array}$
\vspace{.4cm}

$
 \small
\begin{array}{llc}
   {} & {T'_{2i} :} & {(64+2i,29+2i,52+2i,63+2i,2i)(49+2i,32+2i,2i,63+2i,6+2i)}  \\
   {T_{2i} } & {} & {}  \\
   {} & {T''_{2i} :} & {(64+2i,29+2i,52+2i,2i,63+2i)(49+2i,32+2i,63+2i,2i,6+2i)}  \\
   {} &  & {0\le i \le 36};
\end{array}$
\vspace{.4cm}

$
 \small
\begin{array}{llc}
   {} & {T'_{3i} :} & {(2 + 2i,15 + 2i,\infty ,59 + 2i,2i)(22 + 2i,15 + 2i,2 + 2i,42 + 2i,34 + 2i)}  \\
   {T_{3i} } & {} & {}  \\
   {} & {T''_{3i} :} & {(15 + 2i,2 + 2i,\infty ,59 + 2i,2i)(22 + 2i,2 + 2i,15 + 2i,42 + 2i,34 + 2i)}  \\
   {} &  & {0\le i \le 36};
\end{array}$
\vspace{.4cm}

$
 \small
\begin{array}{llc}
   {} & {R'_{1i} :} & {(27 + 2i,48 + 2i,2i,30 + 2i,2 + 2i)(48 + 2i,27 + 2i,23 + 2i,17 + 2i,65 + 2i)}  \\
   {} & {} & {(17 + 2i,23 + 2i,5 + 2i,60 + 2i,56 + 2i)}  \\
   {R_{1i} } & {} & {}  \\
   {} & {R''_{1i} :} & {(48 + 2i,27 + 2i,2i,30 + 2i,2 + 2i)(27 + 2i,48 + 2i,17 + 2i,23 + 2i,65 + 2i)}  \\
   {} & {} & {(23 + 2i,17 + 2i,5 + 2i,60 + 2i,56 + 2i)}  \\
   {} & {} & {0\le i \le 36};
\end{array}$
\vspace{.4cm}

$
 \small
\begin{array}{llc}
   {} & {R'_{2i} :} & {(22 + 2i,58 + 2i,45 + 2i,67 + 2i,2i)(67 + 2i,45 + 2i,69 + 2i,11 + 2i,43 + 2i)}  \\
   {} & {} & {(8 + 2i,67 + 2i,58 + 2i,3 + 2i,66 + 2i)}  \\
   {R_{2i} } & {} & {}  \\
   {} & {R''_{2i} :} & {(22 + 2i,67 + 2i,58 + 2i,45 + 2i,2i)(45 + 2i,67 + 2i,69 + 2i,11 + 2i,43 + 2i)}  \\
   {} & {} & {(8 + 2i,58 + 2i,67 + 2i,66 + 2i,3 + 2i)}  \\
   {} & {} & {0\le i \le 36};
\end{array}$
\vspace{.4cm}

$
 \small
\begin{array}{llc}
   {} & {R'_{3i} :} & {(1 + 2i,2 + 2i,37 + 2i,9 + 2i,29 + 2i)(72 + 2i,27 + 2i,32 + 2i,13 + 2i,22 + 2i)}  \\
   {} & {} & {(8 + 2i,29 + 2i,9 + 2i,22 + 2i,13 + 2i)}  \\
   {R_{3i} } & {} & {}  \\
   {} & {R''_{3i} :} & {(1 + 2i,2 + 2i,37 + 2i,29 + 2i,9 + 2i)(72 + 2i,27 + 2i,32 + 2i,22 + 2i,13 + 2i)}  \\
   {} & {} & {(8 + 2i,9 + 2i,29 + 2i,13 + 2i,22 + 2i)}  \\
   {} & {} & {0\le i \le 36}.
\end{array}$
\vspace{.1cm}

 Therefore by help of the above six groups of  directed trades  we can
obtain
the set of  intersections  of design $D$ and we can   deduce $I_D(75)=J_D(75)$.\\\\
${\bullet ~~~ I_D(81)=J_D(81)}$.\\
 Let $D$ be a   $2$-$(81,5,1)$DD on
the set $V=\{0,1,...,80\}$, with  the below  base blocks \cite{Diff}.
$$\begin{array}{llll}
     (0,1,12,5,26) & (0,2,40,10,64) & (0,3,47,18,53) & (0,9,32,48,68) \\
     (1,0,70,77,56) & (2,0,43,73,19) & (3,0,37,66,31) & (9,0,58,42,22)
\end{array}
$$
In directed  design $D$ there exist $324$   disjoint directed  trades of volume $2$ and  least a directed trade $K$ of
volume $3$.
 %\vspace{.1cm}
\vspace{.4cm}
$$
\small
\begin{array}{llc}
   {} & {T'_{1i}:} & {(i,1+i,12+i,5+i,26+i)(1+i,i,70+i,77+i,56+i)}  \\
   T_{1i} & {} & {}  \\
   {} & {T''_{1i}:} & {(1+i,i,12+i,5+i,26+i)(i,1+i,70+i,77+i,56+i)}  \\
    & {} & {0\leq i \leq 80};
\end{array}$$
\vspace{.4cm}
$$
\small
\begin{array}{lrc}
   {} & {T'_{2i}:} & {(i,2+i,40+i,10+i,64+i)(2+i,i,43+i,73+i,19+i)}  \\
   T_{2i} & {} & {}  \\
   {} & {T''_{2i}:} & {(2+i,i,40+i,10+i,64+i)(i,2+i,43+i,73+i,19+i)}  \\
    & {} & {0\leq i \leq 80};
\end{array}$$
\vspace{.4cm}
$$
\small
\begin{array}{lrc}
   {} & {T'_{3i}:} & {(i,3+i,47+i,18+i,53+i)(3+i,i,37+i,66+i,31+i)}  \\
   T_{3i} & {} & {}  \\
   {} & {T''_{3i}:} &  {(3+i,i,47+i,18+i,53+i)(i,3+i,37+i,66+i,31+i)}  \\
    & {} & {0\leq i \leq 80};
\end{array}$$
\vspace{.4cm}
$$
\small
\begin{array}{lrc}
   {} & {T'_{4i}:} & {(i,9+i,32+i,48+i,68+i)(9+i,i,58+i,42+i,22+i)}  \\
   T_{4i} & {} & {}  \\
   {} & {T''_{4i}:} &  {(9+i,i,32+i,48+i,68+i)(i,9+i,58+i,42+i,22+i)}  \\
    & {} & {0\leq i \leq 80};
\end{array}$$
\vspace{.4cm}
$$
\small
\begin{array}{lrc}
   {} & {K':} & {(0,1,12,5,26)(1,0,70,77,56)(65,66,77,70,10)}  \\
   K & {} & {}  \\
   {} & {K'':} &  {(1,0,12,5,26)(0,1,77,70,56)(65,66,70,77,10)} \\
    & {} & {0\leq i \leq 80}.
\end{array}$$
\vspace{.1cm}
 Therefore by help of the  above four groups of directed  trades, we can
obtain
the set of  intersections of design  $D$ and   can   deduce $I_D(81)=J_D(81)$.\\\\
${\bullet ~~~ I_D(95)=J_D(95)}.$\\
Let $D$ be a  $2$-$(95,5,1)$DD~ on the set  $~V = {Z_{94}} \cup \{\infty \} $,  that is obtained
by developing the below base  blocks  $+1$ $(\bmod ~94)$. (This directed design is obtained from a super simple
$2$-$(95,5,2)$BD with a suitable ordering on its base blocks
\cite{KCRW} ).
$$
\begin{array}{llll}
  (4,79,61,1,74) & (42,47,44,52,37) & (41,66,45,32,91) & (3,72,40,52,9) \\
  (85,35,74,1,58) & (66,1,52,44,88) & (2,0,66,41,48) & (19,52,40,0,90) \\
  (20,37,48,54,21) & (61,81,55,19,83) & (48,57,55,81,1) & (81,57,5,84,17) \\
  (1,83,48,37,66) & (9,1,10,62,81) & (18,58,29,1,9) & (80,6,41,10,9) \\
  (0,53,49,68,84) & (6,36,62,49,0) & (53,0,\infty,15,14) &
\end{array}
$$
\vspace{.1cm}
This $2$-$(95,5,1)$DD contains $235$ disjoint directed trades
 of volume $2$ and  $141$ disjoint directed  trades of volume $3$.
\vspace{.1cm}
$$
\begin{array}{lcc}
   & T'_{1i}:(4+2i,79+2i,61+2i,1+2i,74+2i)(85+2i,35+2i,74+2i,1+2i,58+2i) &  \\
  T_{1i} &  &  \\
   & T''_{1i}:(4+2i,79+2i,61+2i,74+2i,1+2i)(85+2i,35+2i,1+2i,74+2i,58+2i) & \\
   &  1 \leq i \leq 47;&
\end{array}$$
$$
\begin{array}{lcc}
   & T'_{2i}:(42+2i,47+2i,44+2i,52+2i,37+2i)(66+2i,1+2i,52+2i,44+2i,88+2i) &  \\
  T_{2i} &  &  \\
   & T''_{2i}:(42+2i,47+2i,52+2i,44+2i,37+2i)(66+2i,1+2i,44+2i,52+2i,88+2i) & \\
   &  1 \leq i \leq 47;&
\end{array}$$
$$
\begin{array}{lcc}
   & T'_{3i}:(41+2i,66+2i,45+2i,32+2i,91+2i)(2+2i,2i,66+2i,41+2i,48+2i) &  \\
  T_{3i} &  &  \\
   & T''_{3i}:(66+2i,41+2i,45+2i,32+2i,91+2i)(2+2i,2i,41+2i,66+2i,48+2i) & \\
   &  1 \leq i \leq 47;&
\end{array}$$
$$
\begin{array}{lcc}
   &T'_{4i}: (3+2i,72+2i,40+2i,52+2i,9+2i)(19+2i,52+2i,40+2i,2i,90+2i) &  \\
  T_{4i} &  &  \\
   &T''_{4i}: (3+2i,72+2i,52+2i,40+2i,9+2i)(19+2i,40+2i,52+2i,2i,90+2i) & \\
   &  1 \leq i \leq 47;&
\end{array}$$
$$
\begin{array}{lcc}
   & T'_{5i}:(20+2i,37+2i,48+2i,54+2i,21+2i) (1+2i,83+2i,48+2i,37+2i,66+2i)&  \\
  T_{5i} &  &  \\
   & T''_{5i}:(20+2i,48+2i,37+2i,54+2i,21+2i) (1+2i,83+2i,37+2i,48+2i,66+2i) & \\
   &  1 \leq i \leq 47;&
\end{array}$$
$$
\begin{array}{lcc}
   &R'_{1i}: (61+2i,81+2i,55+2i,19+2i,83+2i)(48+2i,57+2i,55+2i,81+2i,1+2i) &  \\
   &  (81+2i,57+2i,5+2i,84+2i,17+2i)&  \\
R_{1i}   &  &  \\
   & R''_{1i}:(61+2i,55+2i,81+2i,19+2i,83+2i)(48+2i,81+2i,57+2i,55+2i,1+2i) & \\
   & (57+2i,81+2i,5+2i,84+2i,17+2i)&  \\
   &  1 \leq i \leq 47;&
\end{array}$$
$$
\begin{array}{lcc}
   & R'_{2i}:(9+2i,1+2i,10+2i,62+2i,81+2i)(18+2i,58+2i,29+2i,1+2i,9+2i)&  \\
   &(80+2i,6+2i,41+2i,10+2i,9+2i) &  \\
R_{2i}   &  &  \\
   & R''_{2i}:(1+2i,10+2i,9+2i,62+2i,81+2i)(18+2i,58+2i,29+2i,9+2i,1+2i) &  \\
   & (80+2i,6+2i,41+2i,9+2i,10+2i) &  \\
   &  1 \leq i \leq 47;&
\end{array}$$
$$
\begin{array}{lcc}
   &R'_{3i}: (i,53+2i,49+2i,68+2i,84+2i)(6+2i,36+2i,62+2i,49+2i,2i) &  \\
   &  (53+2i,2i,\infty,15+2i,14+2i)&  \\
 R_{3i}  &  &  \\
   &R''_{3i}: (53+2i,49+2i,2i,68+2i,84+2i)(6+2i,36+2i,62+2i,i,49+2i) & \\
   & (2i,53+2i,\infty,15+2i,14+2i) &  \\
   &  1 \leq i \leq 47.&
\end{array}$$
Therefore according  to the above  trades that there exist in design
$D$, we can obtain
the set of intersections of  $(95,5,1)DD$ and   can   deduce $I_D(95)=J_D(95)$.\\\\
${\bullet ~~~ I_D(115)=J_D(115)}$.\\
Let $D$ be a  $2$-$(115,5,1)$DD~ on the set  $~V = {Z_{ 104}} \cup
\{ {\infty _0},{\infty _1},...,{\infty _{10}}\} $,  that is
obtained by developing the blocks below.
 The design $D$ contain two groups block, the first group blocks is obtained by adding the value $+2$ $(\mathrm{mod ~104})$ to the below blocks and in this  blocks $\infty_0$ is changed $\infty_1$ when adding $+2$ $(\bmod 104)$.
  %each the below block forms 2 blocks.
 $$
 \begin{array}{cc}
   (9,61,56,4,\infty_0) & (61,9,36,88,\infty_1) \\
   (13,65,60,8,\infty_0) & (65,13,40,92,\infty_1) \\
   (17,69,64,12,\infty_0) & (69,17,44,96,\infty_1) \\
   (21,73,68,16,\infty_0) & (73,21,48,100,\infty_1) \\
   (25,77,72,20,\infty_0) & (77,25,52,0,\infty_1) \\
   (29,81,76,24,\infty_0) & (81,29,4,56,\infty_1) \\
   (33,85,80,28,\infty_0) & (85,33,8,60,\infty_1) \\
   (37,89,84,32,\infty_0) & (89,37,12,64,\infty_1) \\
   (41,93,88,36,\infty_0) & (93,41,16,68,\infty_1) \\
   (45,97,92,40,\infty_0) & (97,45,20,72,\infty_1) \\
   (49,101,96,44,\infty_0) & (101,49,24,76,\infty_1) \\
   (53,1,100,48,\infty_0) & (1,53,28,80,\infty_1) \\
   (57,5,0,52,\infty_0) & (5,57,32,84,\infty_1)
 \end{array}
 $$
 The second group blocks is obtained by developing the below base blocks $+2~(\mathrm{mod}~104)$, and in the first block  $\infty_0$ is replaced by $\infty_1$ when adding any value $\equiv 2 (\bmod~4)$. Finally
form a $2$-$(11,5,1)$DD, on the set $~\{{\infty _0},{\infty
_1},...,{\infty _{10}}\}$.
$$
\begin{array}{llllll}
    (0,3,67,25,17) & (6,103,59,0,33) & (0,59,49,96,28) & (0,58,88,65,84) & (0,21,\infty _3 ,93,92) & (75,0,\infty _5 ,103,62) \\
    (16,67,3,84,93) & (24,87,\infty_2,103,6) & (83,29,49,59,12) & (53,44,88,58,78) & (92,93,1,78,89) & (0,75,\infty_4,64,83) \\
    (\infty_0,47,2,4,41) & (4,2,26,49,44) & (23,32,\infty_6,4,47) & (5,3,47,51,44) & (3,5,87,58,41) & (55,2,\infty_7,58,87) \\
    (21,39,52,3,9) & (76,41,\infty_9,3,52) & (42,37,\infty_8,9,52) & (0,4,50,12,82) & (100,44,12,85,50) & (29,94,\infty_{10},50,85)
  \end{array}
$$
 This $2$-$(115, 5, 1)$DD contains $338$  disjoint directed  trades of volume $2$ , $208$  disjoint directed trades of volume
$3$.
\vspace{.1cm}

$
\begin{array}{rlc}
   {} & {S'_{1i} :} & {(2i,3+2i,67+2i,25+2i,17+2i)(16+2i,67+2i,3+2i,84+2i,93+2i)}  \\
   {S_{1i} } & {} & {}  \\
   {} & {S''_{1i} :} & {(2i,67+2i,3+2i,25+2i,17+2i)(16+2i,3+2i,67+2i,84+2i,93+2i)}  \\
   {} & {} & {0 \le i \le 51};
\end{array}$
\vspace{.1cm}

 $
\begin{array}{rlc}
   {} & {S'_{2i} :} & {(6+2i,103+2i,59+2i,2i,33+2i)(24+2i,87+2i,\infty_2,103+2i,6+2i)}  \\
   {S_{2i} } & {} & {}  \\
   {} & {S''_{2i} :} & {(103+2i,6+2i,59+2i,2i,33+2i)(24+2i,87+2i,\infty_2,6+2i,103+2i)}   \\
   {} & {} & {0 \le i \le 51};
\end{array}$
\vspace{.1cm}

 $
\begin{array}{rlc}
   {} & {S'_{3i} :} & {(2i,59+2i,49+2i,96+2i,28+2i)(83+2i,29+2i,49+2i,59+2i,12+2i)}  \\
   {S_{3i} } & {} & {}  \\
   {} & {S''_{3i} :} & {(2i,49+2i,59+2i,96+2i,28+2i)(83+2i,29+2i,59+2i,49+2i,12+2i)}  \\
   {} & {} & {0 \le i \le 51};
\end{array}$
\vspace{.1cm}

 $
\begin{array}{rlc}
   {} & {S'_{4i} :} & {(2i,58+2i,88+2i,65+2i,84+2i)(53+2i,44+2i,88+2i,58+2i,78+2i)}  \\
   {S_{4i} } & {} & {}  \\
   {} & {S''_{4i} :} & {(2i,88+2i,58+2i,65+2i,84+2i)(53+2i,44+2i,58+2i,88+2i,78+2i)}   \\
   {} & {} & {0 \le i \le 51};
\end{array}$
\vspace{.1cm}

$
\begin{array}{rlc}
   {} & {S'_{5i} :} & {(2i,21+2i,\infty _3 ,93+2i,92+2i)(92+2i,93+2i,1+2i,78+2i,89+2i)}  \\
   {S_{5i} } & {} & {}  \\
   {} & {S''_{5i} :} & {(2i,21+2i,\infty _3 ,92+2i,93+2i)(93+2i,92+2i,1+2i,78+2i,89+2i)}  \\
   {} & {} & {0 \le i \le 51};
\end{array}$
\vspace{.1cm}

$
\begin{array}{rlc}
   {} & {S'_{6i} :} & {(75+2i,2i,\infty _5 ,103+2i,62+2i)(2i,75+2i,\infty_4,64+2i,83+2i)}  \\
   {S_{6i} } & {} & {}  \\
   {} & {S''_{6i} :} & {(2i,75+2i,\infty _5 ,103+2i,62+2i)(75+2i,2i,\infty_4,64+2i,83+2i)}  \\
   {} & {} & {0 \le i \le 51};
\end{array}$
\vspace{.1cm}

$
\begin{array}{rlc}
   {} & {R'_{1i} :} & {(\infty_0,47+2i,2+2i,4+2i,41+2i)(4+2i,2+2i,26+2i,49+2i,44+2i)}  \\
   {} & {} & {(23+2i,32+2i,\infty_6,4+2i,47+2i)}  \\
   {R_{1i} } & {} & {}  \\
   {} & {R''_{1i} :} & {(\infty_0,4+2i,47+2i,2+2i,41+2i)(2+2i,4+2i,26+2i,49+2i,44+2i)}  \\
   {} & {} & {(23+2i,32+2i,\infty_6,47+2i,4+2i)}  \\
   {} & {} & {0 \le i \le 51};
\end{array}$
\vspace{.1cm}

$
\begin{array}{rlc}
   {} & {R'_{2i} :} & {(5+2i,3+2i,47+2i,51+2i,44+2i)(3+2i,5+2i,87+2i,58+2i,41+2i)}  \\
   {} & {} & {(55+2i,2+2i,\infty_7,58+2i, 87+2i)}  \\
   {R_{2i} } & {} & {}  \\
   {} & {R''_{2i} :} & {(3+2i,5+2i,47+2i,51+2i,44+2i)(5+2i,3+2i,58+2i,87+2i,41+2i)}   \\
   {} & {} & {(55+2i,2+2i,\infty_7,87+2i,58+2i)}  \\
   {} & {} & {0 \le i \le 51};
\end{array}$
\vspace{.1cm}

$
\begin{array}{rlc}
   {} & {R'_{3i} :} & {(21+2i,39+2i,52+2i,3+2i,9+2i)  (76+2i,41+2i,\infty_9,3+2i,52+2i)  }  \\
   {} & {} & {(42+2i,37+2i,\infty_8,9+2i,52+2i)}  \\
   {R_{3i} } & {} & {}  \\
   {} & {R''_{3i} :} & {(21+2i,39+2i,3+2i,9+2i,52+2i)  (76+2i,41+2i,\infty_9,52+2i,3+2i)  }   \\
   {} & {} & {(42+2i,37+2i,\infty_8,52+2i,9+2i)}  \\
   {} & {} & {0 \le i \le 51};
\end{array}$
\vspace{.1cm}

$
\begin{array}{rlc}
   {} & {R'_{4i}:} & {(2i,4+2i,50+2i,12+2i,82+2i) (100+2i,44+2i,12+2i,85+2i,50+2i)}  \\
   {} & {} & {(29+2i,94+2i,\infty_{10},50+2i,85+2i)}  \\
   R_{4i} & {} & {}  \\
   {} & {R''_{4i}:} & {(2i,4+2i,12+2i,50+2i,82+2i) (100+2i,44+2i,50+2i,12+2i,85+2i)}  \\
   {} & {} & {(29+2i,94+2i,\infty_{10},85+2i,50+2i)}  \\
   {} & {} & {0 \le i \le 51}.
\end{array}$
\vspace{.1cm}

Therefore by help of the above directed trades that there exist in design $D$ we can
obtain
the set of intersections $(115,5,1)DD$ and   can   deduce $I_D(115)=J_D(115)$.\\
\\
${\bullet ~~~I_D(v)=J_D(v)}$, for $v=135,215,335$.\\
According to  Proposition \ref{12}  for  each of values
$v=135,215,335$ there exists a $\{5,6\}-GDD$ of type $5^{\frac{{v
- 5}}{{10}}} 2$. Now by  Lemma \ref{4} and the fact that $I_G
(10) = \{0,1,...,6,8\}$, $I_G (12) = \{0,1,...,10,12\}$,
$I_D(5)=J_D(5)$ and $\{ 0,1,2,3,11\}  \subseteq I_D (11)$, we can
deduce
$I_D(v)=J_D(v)$.\\\\
${ \bullet ~~~I_D(v)=J_D(v)}$, for $v=155,235$.\\
According to  Proposition \ref{12} for each of  values
$v=155,235$ there exists a $\{5,6\}-GDD$ of type $5^{\frac{{v -
25}}{{10}}} 12$. Now by  Lemma \ref{4} and the fact that $I_G
(10) = \{0,1,...,6,8\}$, $I_G (12) = \{0,1,...,10,12\}$,
$I_D(25)=J_D(25)$ and $\{ 0,1,2,3,11\} \subseteq I_D (11)$, we can
deduce
$I_D(v)=J_D(v)$ .\\
\end{document}